\documentclass{amsart}

\usepackage[utf8]{inputenc}
\usepackage[T1]{fontenc}

\usepackage[colorlinks=true]{hyperref}
\hypersetup{allcolors=[rgb]{0.1,0.1,0.4}}

\title{The Canonical 2-Gerbe of a Holomorphic Vector Bundle}

\author{Markus Upmeier}
\address{Universit\"atsstrasse 14, 86159 Augsburg, Germany}
\email{markus.upmeier@math.uni-augsburg.de}

\thanks{The author was partially funded by the Deutsche Forschungsgemeinschaft (DFG, German Research Foundation) -- UP 85/2-1, UP 85/3-1.}

\usepackage[all]{xy}

\usepackage{tikz}
\usetikzlibrary{matrix,arrows,calc}
\usetikzlibrary{decorations.markings}

\usepackage{mathrsfs}
\usepackage{eucal}
\usepackage{amsmath,amsfonts}
\usepackage{latexsym,amssymb}
\usepackage{stmaryrd}

\newtheorem{theorem}{Theorem}
\newtheorem{lemma}[theorem]{Lemma}
\newtheorem{corollary}[theorem]{Corollary}
\newtheorem{proposition}[theorem]{Proposition}

\theoremstyle{definition}
\newtheorem{definition}[theorem]{Definition}

\theoremstyle{remark}
\newtheorem{remark}[theorem]{Remark}

\def\N{\mathbb{N}}
\def\Z{\mathbb{Z}}
\def\R{\mathbb{R}}
\def\C{\mathbb{C}}

\DeclareMathOperator\incl{incl}
\DeclareMathOperator\id{id}
\DeclareMathOperator\pr{pr}
\def\can{{\mathrm{can}}}
\DeclareMathOperator\im{im}
\DeclareMathOperator\spec{spec}

\def\GL{\mathrm{GL}}

\def\O{{\mathcal{O}}}
\def\scrG{{\mathcal{G}}}

\begin{document}

\maketitle

\begin{abstract}
For each holomorphic vector bundle we construct a holomorphic bundle 2-gerbe that geometrically represents its second Beilinson--Chern class. Applied to the cotangent bundle, this may be regarded as a higher analogue of the canonical line bundle in complex geometry. Moreover, we exhibit the precise relationship between holomorphic and smooth gerbes. For example, we introduce an Atiyah class for gerbes and prove a Koszul--Malgrange type theorem.
\end{abstract}

\section{Introduction}

A fundamental object associated to complex manifolds $X$ is its canonical line bundle $\Lambda^\mathrm{top} (T^*X)$. The present paper deals with an extension of this concept. Recall that the canonical bundle may be viewed as a representative of the first Chern class. In Theorem~\ref{thm:Canonical2Gerbe} we construct a geometric representative of the second Be{\u\i}linson--Chern class, a refinement introduced in \cite{MR760999} that takes the holomorphic structure into account. This is a $2$-gerbe and may be computed from the eigenvalues and eigenspaces of chart transition functions, just like the canonical bundle depends only on the product of these eigenvalues.

We are motivated by the relationship between $2$-gerbes and line bundles on the space of curves, via transgression \cite{MR2174418,MR1667679,MR2362847,MR2980505}. According to the paradigm of string geometry, structure on the infinite-dimensional loop space may be studied through higher-categorical, finite-dimensional structure on the original space. This is the viewpoint of the present paper. 
The $2$-gerbe will be constructed within the framework of bundle gerbes, introduced in \cite{MR1405064} and further studied  in \cite{MR1794295,MR2681698,MR2032513}. It is obtained by pulling back a `tautological' \emph{multiplicative} $1$-gerbe on $\GL(n,\C)$.
Having applications to complex geometry in mind, we shall be particularly concerned with the holomorphic structure on our gerbe. Even though the definition of our gerbe on $\GL(n,\C)$ is similar to that of $\mathrm{U}(n)$ in \cite{MR2463811}, this forces us to apply different arguments based on zero-counting integrals.

To set this work apart from related results, we now briefly review the existing literature. In \cite{MR1388845,MR1291698,MR2362847} tautological gerbes are defined as Dixmier--Douady sheaves of groupoids, but the constructions involve infinite-dimensional spaces. This makes them less useful from the point of view of string geometry. We shall construct instead a very explicit and finite-dimensional model. Moreover, we have \emph{smooth} gerbes over compact, simple, simply-connected Lie groups \cite{MR2026898}, certain quotients therefore \cite{MR2078218}, for $\mathrm{SU}(n)$ \cite{MR1945806}. Of course, $\GL(n,\C)$ is of none of these types. For compact semi-simple Lie groups $G$, a smooth multiplicative structure on the tautological gerbe on $G$ was constructed by cohomological arguments in \cite{MR2174418} and \cite{MR2610397}.

We now state our main results and give an overview of the present paper.

We begin in Section~\ref{sec:HolomorphicGerbes} by developing further the theory of holomorphic bundle gerbes started in \cite[Section~7]{MR1977885}. In~\ref{ssec:relation-smooth} we study the relationship to smooth gerbes and present an analogue of the Koszul--Malgrange Theorem~\cite{MR0131882}. For holomorphic line bundles, the Atiyah class is the obstruction against holomorphic connections. In \ref{ssec:atiyah} we demonstrate how this idea extends to gerbes. The rest of the section reviews well-known terminology for smooth gerbes \cite{MR1405064,MR1794295,MR2681698,MR2318389} in our holomorphic context.

Our first theorem gives a very explicit and finite-dimensional construction of a holomorphic gerbe $\scrG^\can$. It extends the work \cite{MR2463811} for the smooth Lie group $\mathrm{U}(n)$. However, as $\GL(n,\C)$ is non-compact and in the holomorphic category we require different arguments based on zero-counting integrals. In Theorem~\ref{main-thm-1} we prove:

\begin{theorem}\label{thm:HolomorphicGerbe}
$\scrG^\can=(\pi\colon Y \rightarrow \GL(n,\C), L, m)$ defines a holomorphic gerbe. Its Dixmier--Douady class $\mathrm{DD}(\scrG^\can)$ is the generator of $H^2(\GL(n,\C), \O^*)$.
\end{theorem}

The complex Lie group $\GL(n,\C)$ is a Stein group. The techniques in Section~\ref{sec:CohomStein} for Stein manifolds allow us to show in Section~\ref{sec:MultiplicativeStructure} that the existence of a holomorphic multiplicative structure is a purely topological problem:

\begin{theorem}\label{theorem:MainMultiplicativity}
Let $G$ be a Stein group and let $\scrG$ be a holomorphic gerbe on $G$ with holomorphic connection. Then $\scrG$ admits a holomorphic multiplicative structure with connection precisely when the \emph{topological} Dixmier--Douady class $\mathrm{DD}(\scrG)\in H^3(G;\Z)$ is in the image of the transgression map $H^4(BG;\Z)\rightarrow H^3(G;\Z)$.
\end{theorem}

This shows that for Stein groups, the problem may be reduced to Waldorf's theory \cite{MR2610397}. In particular, every holomorphic gerbe with connection on $\GL(n,\C)$ admits a multiplicative structure (Corollary~\ref{GLNcorollary}). Following the proof, we also give a more explicit description. It is special to our treatment that we use Stein spaces to study the Be{\u\i}linson--Chern classes, instead of Deligne's theory of mixed Hodge structures \cite{MR0498552} used in \cite{MR760999,MR944991}.

Sections~\ref{ssec:2gerbe} and \ref{ssec:DD2gerbe} introduce a notion of {$2$-gerbe} which is slightly weaker than in \cite{MR2032513}. We also review the definition of the Dixmier--Douady class of a $2$-gerbe. In the holomorphic context, there is a subtle point with the choice of coverings we feel is not adequately considered in the existing literature.

For each holomorphic vector bundles $E \rightarrow X$ we present in Section~\ref{ssec:canonical2gerbe} a construction for a $2$-gerbe $\mathfrak{G}(E)$. We show:

\begin{theorem}\label{thm:Canonical2Gerbe}
The associated $2$-gerbe has the following properties:
\begin{enumerate}
\item
$\mathfrak{G}(f^*E) \cong f^*\mathfrak{G}(E)$ \textup(functorial\textup)
\item
The topological Dixmier--Douady class is $c_2(E)$.
\item
For $X$ algebraic, $\mathrm{DD}(\mathfrak{G}(E))=c_2^B(E)$ is the Be{\u\i}linson--Chern class of $E$.
\end{enumerate}
\end{theorem}

Applied to $E=T^*X$ we obtain the \emph{canonical $2$-gerbe} of a complex manifold $X$. In the last Section~\ref{sec:appl}, we illustrate an application that relies on the newly established \emph{holomorphic structure} on the canonical $2$-gerbe.

\subsection*{Notation}

For iterated fiber products we use the notation ${Y^{[i]}=Y\times_X \cdots \times_X Y}$ and similarly for maps over $X$. By $\pr_{ijk\cdots}$ we mean the projections onto the indicated factors. The sheaf of holomorphic $k$-forms is $\Omega^k$ and $\O^*$ is the sheaf of nowhere vanishing holomorphic functions. Unless stated otherwise, \emph{all bundles and bundle maps below are assumed to be holomorphic}.

\subsection*{Acknowledgements} I thank Joel Fine for many discussions on these results.

\section{Holomorphic Bundle Gerbes and Connections}\label{sec:HolomorphicGerbes}

In this section we review well-established concepts for smooth gerbes in the holomorphic
context, see~\cite{MR1977885}. Also some new results are developed, such a Koszul--Malgrange Theorem for gerbes in \ref{ssec:relation-smooth} or a generalization of the Atiyah class in \ref{ssec:atiyah}. Brylinski's holomorphic gerbes \cite{MR2362847}, sheaves of groupoids with band $\O^*$, are equivalent to the holomorphic bundle gerbes considered here.

\subsection{Holomorphic Gerbes}\label{ssec:holomorphic-gerbes}
Smooth gerbes represent classes in ${H^2(X, \underline{\C}^*)\cong H^3(X;\Z)}$ geometrically. In the simplest approach, one uses \v{C}ech cocycles for an open cover of $X$, leading to Hitchin--Chatterjee gerbes \cite{Chatterjee}. Below, we shall define holomorphic gerbes on complex manifolds $X$. These objects are designed to represent  cohomology classes in $H^2(X, \O_X^*)$. We will adopt the `bundle gerbe' approach from \cite{MR1405064} which uses descent theory to avoid artificial choices of an open cover.

\begin{definition}\label{def:HolomorphicGerbe}
A \emph{holomorphic gerbe} $\scrG$ on $X$ is a triple $(\pi,L,m)$ consisting of a
holomorphic submersion $\pi\colon Y\rightarrow X$, 
holomorphic line bundle $L\rightarrow Y^{[2]}$, and
bundle isomorphism ${m\colon \pr_{12}^*L \otimes \pr_{23}^*L \rightarrow \pr_{13}^*L}$.
Denoting by $L_{y_1,y_2}$ the fibers of $L$, we get `multiplication maps' ${m\colon L_{y_1,y_2}\otimes L_{y_2,y_3} \rightarrow L_{y_1,y_3}}$, which are required to be associative \textup(see~\textup{\cite[Definition~4.1]{MR2681698}}\textup).
\end{definition}

\begin{definition}
A \emph{connection} on ${\scrG=(Y,L,m)}$ consists of a holomorphic connection on $L$ \textup(see~\textup{\cite[Definition~4.2.17]{MR2093043}}\textup) preserved by the multiplication $m$.
A \emph{curving} on a gerbe with connection is a form $f \in \Omega^2(Y)$ related to the curvature of $L$ by
\begin{equation}\label{EqnCurvingCurvature}
	\pr_2^*f - \pr_1^*f = R(\nabla^L).
\end{equation}
The \emph{3-curvature} is then the unique $\rho=R(\scrG,\nabla^L,f) \in \Omega^3(X)$ with ${\pi^*\rho = df}$.
\end{definition}
%

In contrast to the smooth category, holomorphic connections are rare (see for example Corollary~\ref{cor-kaehler} below). The precise obstructions to their existence will be explained in Section~\ref{ssec:atiyah} with the help of Deligne cohomology. There is also the weaker notion of \emph{compatible connection} on $\scrG$. It consists of a connection on $L$ only \emph{compatible} with the holomorphic structure (see~\cite{MR2093043}). The curving is then a form of type $(2,0)+(1,1)$.

\subsection{Relationship to Smooth Gerbes}\label{ssec:relation-smooth}

Holomorphic gerbes \emph{with} holomorphic connections may equivalently be described by connective data with certain properties. Recall the obvious variant of Definition~\ref{def:HolomorphicGerbe} in the category of smooth manifolds \cite{MR2681698}.

\begin{theorem}
Let $\scrG=(\pi\colon Y\rightarrow X,L,m)$ be a smooth gerbe whose projection $\pi$ is a holomorphic submersion.
Suppose it is equipped with a connection and curving $f$ of type $(2,0)$ whose $3$-curvature $\rho$ is of type $(3,0)$.
Then there exists a canonical structure of holomorphic gerbe on $\scrG$ with holomorphic connection.
\end{theorem}

\begin{proof}
The $(0,1)$-part $(\nabla^L)^{(0,1)}$ of the smooth connection determines a Cauchy-Riemann operator on $L$. By equation \eqref{EqnCurvingCurvature} and our assumption on $\rho$, the curvature of the connection is of type $(2,0)$ which shows in particular that $(\nabla^L)^{(0,1)}\circ (\nabla^L)^{(0,1)} = 0$ is flat. Therefore the Theorem of Koszul--Malgrange \cite{MR0131882} shows that this Cauchy--Riemann operator defines a holomorphic structure on $L$. The multiplication $m$ preserves the $(0,1)$-part of the connection and is thus holomorphic. As the $3$-curvature $\rho$ of type $(3,0)$ satisfies ${\pi^*\rho = df}$, we have $\bar{\partial} f=0$, so the curving $f$ is holomorphic as well.
\end{proof}

Note that, conversely, a holomorphic gerbe with holomorphic connection satisfies the hypotheses of the theorem.
A similar statement applies to holomorphic gerbes with connection, but without curving --- the hypothesis is then that $R(\nabla^L)$ is of type $(2,0)$. Moreover, we note that the weaker notion of compatible connection with curving on $\scrG$ corresponds to a smooth connection whose curving has type $(2,0)+(1,1)$.

\subsection{Deligne Cohomology. Dixmier--Douady Classes}\label{ssec:DD}

The \emph{Deligne complex} $\Z(p)_D$ is the following complex of sheaves (see~\cite[Section~1.5]{MR2362847})
\begin{equation*}
\begin{tikzpicture}
\matrix (m) [matrix of math nodes, row sep=0em,
    column sep=2.5em, text height=1.5ex, text depth=0.25ex]{
	\Z(p)_D\colon 0&\O^*&\Omega^1&\cdots&\Omega^{p-1}.\\
};
\draw[->, font=\scriptsize] (m-1-1) -- (m-1-2);
\draw[->, font=\scriptsize] (m-1-2) -- node [above] {$d\log$} (m-1-3);
\draw[->, font=\scriptsize] (m-1-3) -- node [above] {$d$} (m-1-4);
\draw[->, font=\scriptsize] (m-1-4) -- node [above] {$d$} (m-1-5);
\end{tikzpicture}
\end{equation*}
The \emph{Deligne cohomology} groups $H^q(X, \Z(p)_D)$ are the hypercohomology groups of this complex. For $X$ paracompact they are isomorphic to the \v{C}ech hypercohomology groups $\check{H}^q(X, \Z(p)_D)$ (Godement's Theorem), which are convenient to construct actual classes.

Of course $H^n(X,\Z(1)_D)=H^{n-1}(X,\O^*)$. The higher Deligne groups relate to connective structure. Thus $H^2(X, \Z(2)_D)$ is the group of holomorphic line bundles on $X$ with holomorphic connection, see~\cite{MR1114212}. We now extend this classification to gerbes.

\begin{lemma}[{\cite[Lemma~7.2.3.5]{MR2522659}}]\label{lem:IteratedCover}
Let $\mathcal{U}=\{U_\alpha\}$ be an open covering of a paracompact space $X$, $k\in \N$.
Given an open cover $\mathcal{V}^{\alpha_0\cdots \alpha_k}$ of each $(k+1)$-fold intersection $U_{\alpha_0\cdots \alpha_k}$, there exists a refinement $\{W_\beta\}$ of $\mathcal{U}$ with the property that each $(k+1)$-fold intersection $W_{\beta_0\cdots \beta_k}$ is a subset of some element of $\mathcal{V}^{\alpha_0\cdots \alpha_k}$.
\end{lemma}

\begin{corollary}\label{cor:GerbeTrivial}
Let $\scrG=(\pi,L,m)$ be a holomorphic gerbe. Then we find
arbitrarily fine open covers $\{U_\alpha\}$ of $X$ which admit
holomorphic sections ${s_\alpha\colon U_\alpha \rightarrow Y}$ of $\pi$ and holomorphic trivializations $\sigma_{\alpha\beta}$ of $(s_\alpha,s_\beta)^*L$.
\end{corollary}

We now come to the Dixmier--Douady class, which naturally lies in $H^3(X;\Z(p)_D)$, $p$ signifying the amount of connective structure on the gerbe.

\begin{definition}
Let $\scrG$ be a holomorphic gerbe. Pick an open cover $\{U_\alpha\}$ as in Corollary~\ref{cor:GerbeTrivial}. Define $g_{\alpha\beta\gamma}\in \O^*(U_{\alpha\beta\gamma})$ by
\[
	m(\sigma_{\alpha\beta}, \sigma_{\beta\gamma}) = g_{\alpha\beta\gamma} \cdot\sigma_{\alpha\gamma}.
\]
Associativity of $m$ implies that $(g_{\alpha\beta\gamma})$ is closed. Different trivializations give cohomologous cocycles. Taking the limit over a cofinal sequence of covers, we define
\[
	\mathrm{DD}(\scrG)=[(g_{\alpha\beta\gamma})] \in \check{H}^3(X, \Z(1)_D).
\]
If $\scrG$ is equipped with a holomorphic connection $\nabla$, this class may be refined to
\[
	\mathrm{DD}(\scrG, \nabla)\in \check{H}^3(X, \Z(2)_D).
\]
For this we write the connection on $(s_\alpha,s_\beta)^*L$ as ${d + A_{\alpha\beta}}$. Since $m$ preserves the connection we have $d\log(g_{\alpha\beta\gamma}) = A_{\beta\gamma} - A_{\alpha\gamma} + A_{\alpha\beta}$. Hence
$(g_{\alpha\beta\gamma}, A_{\alpha\beta}) \in \check{C}^3(\{U_\alpha\}, \Z(2)_D)$ is a \v{C}ech cocycle.
If, in addition, we suppose $\scrG$ to be equipped with a curving, then $f_\beta - f_\alpha = dA_{\alpha\beta}$ for $f_\alpha=s_\alpha^*f$, so $(g_{\alpha\beta\gamma}, A_{\alpha\beta}, f_\alpha)$ defines a class $\mathrm{DD}(\scrG, \nabla,f) \in H^3(X, \Z(3)_D)$.
\end{definition}

The exponential map to $H^3(X;\Z)$ maps this class to the topological Dixmier--Douady class $\mathrm{DD}^\mathrm{top}(\scrG)$, defined for example in~\cite{MR1405064,MR2681698}.

\subsection{Obstructions to Holomorphic Connections}\label{ssec:atiyah}

As mentioned, holomorphic connections and curvings cannot always be found. This is measured by the following generalization of the Atiyah class for holomorphic line bundles. The short exact sequence of sheaves
\[
	0\rightarrow\Omega^{p}[-p-1]\rightarrow \Z(p+1)_D \rightarrow \Z(p)_D \rightarrow 0
\]
induces the following exact sequences
\begin{equation*}
\begin{tikzpicture}
\matrix (m) [matrix of math nodes, row sep=0.5em,
    column sep=2.5em, text height=1.5ex, text depth=0.25ex]{
	H^3(X, \Z(2)_D) & H^3(X, \Z(1)_D) & H^2(X, \Omega^1),\\
	H^3(X, \Z(3)_D) & H^3(X, \Z(2)_D) & H^1(X, \Omega^2).\\
};
\draw[->, font=\scriptsize] (m-1-1) -- (m-1-2);\draw[->, font=\scriptsize] (m-1-2) -- node[above]{$B$} (m-1-3);
\draw[->, font=\scriptsize] (m-2-1) -- (m-2-2);\draw[->, font=\scriptsize] (m-2-2) -- node[above]{$C$} (m-2-3);
\end{tikzpicture}
\end{equation*}

\begin{corollary}
Let $\scrG$ be a holomorphic gerbe. Then $\scrG$ admits a holomorphic connection if and only if $B(\mathrm{DD}(\scrG)) = 0$. Assume $\scrG$ is equipped with a holomorphic connection $\nabla$. Then we may find a curving if and only if $C(\mathrm{DD}(\scrG,\nabla))=0$.
\end{corollary}

\begin{proposition}
The images of the classes $B(\scrG)$ and $\mathrm{DD}(\scrG)$ in $H^3(X;\C)$ agree.
\end{proposition}
\begin{proof}
From the definition of the coboundary map, one sees that $B$ maps the \v{C}ech cocycle $g_{\alpha\beta\gamma} \in \O^*(U_{\alpha\beta\gamma})$ to the class of $d\log g_{\alpha\beta\gamma}$ in $H^2(A^1_\mathrm{cl})$.
%
%
%
Therefore the image of $d\log g_{\alpha\beta\gamma}$ in $H^3(X;\C)$ is the image of the Dixmier--Douady class $\delta(g_{\alpha\beta\gamma})$.
\end{proof}

\begin{corollary}\label{cor-kaehler}
For a compact K\"ahler manifold $X$ the topological Dixmier--Douady class $\mathrm{DD}(\scrG) \in H^3(X;\Z)$ of a gerbe with holomorphic connection must be torsion.
\end{corollary}

This illustrates how restrictive the existence of a holomorphic connection is. For example when $H^3(X;\Z)$ is torsion-free, they exist only on gerbes that are topologically trivial.

\subsection{Morphisms of Gerbes}

There is a naive notion (which we do not discuss) of isomorphism with which two gerbes may have the same Dixmier--Douady class without being isomorphic. When wishing to emphasize this, we also call the morphisms below stable isomorphisms, in accordance with usual terminology.

\begin{definition}
Let ${\scrG=(Y, L, m)}$, ${\scrG'=(Y', L', m')}$ be holomorphic gerbes on $X$. A \emph{morphism} $F$ from $\scrG$ to $\scrG'$ is a triple $((\varsigma,\varsigma'), R, \phi)$ of a
submersion ${(\varsigma,\varsigma')}\colon {Z \rightarrow Y \times_X Y'}$,
line bundle ${R\rightarrow Z}$, and
bundle isomorphism
\[
	\phi\colon {(\varsigma^{[2]})^*L \otimes \pr_2^* R \rightarrow \pr_1^*R\otimes (\varsigma'^{[2]})^*L'}.
\]
Hence $\phi$ gives maps ${L_{y_1,y_2}\otimes R_{z_2} \rightarrow R_{z_1}\otimes L'_{y'_1,y'_2}}$ for ${(z_1,z_2)\in Z^{[2]}}$, where ${\varsigma(z_k)=y_k}$, ${\varsigma'(z_k)=y'_k}$. For ${(z_1,z_2,z_3)\in Z^{[3]}}$ we require commutative diagrams:
\begin{equation*}
\begin{tikzpicture}[baseline=(current  bounding  box.center)]
\matrix (m) [matrix of math nodes, column sep=3em, row sep=2em]
{
	L_{y_1,y_2}\otimes L_{y_2,y_3}\otimes R_{z_3}&
	L_{y_1,y_2}\otimes R_{z_2}\otimes L'_{y'_2,y'_3}&
	R_{z_1}\otimes L'_{y'_1,y'_2}\otimes L'_{y'_2,y'_3}\\
	L_{y_1,y_3}\otimes R_{z_3}&
	&
	R_{z_1}\otimes L'_{y'_1,y'_3}\\
};

\draw[->, font=\scriptsize] (m-1-1) -- node [left] {$m\otimes\id$} (m-2-1);
\draw[->, font=\scriptsize] (m-1-3) -- node [right] {$\id\otimes m'$} (m-2-3);

\draw[->, font=\scriptsize] (m-1-1) -- node [above] {$\id\otimes \phi$} (m-1-2);
\draw[->, font=\scriptsize] (m-1-2) -- node [above] {$\phi\otimes\id$} (m-1-3);

\draw[->, font=\scriptsize] (m-2-1) -- node [above] {$\phi$} (m-2-3);
\end{tikzpicture}
\end{equation*}
\end{definition}

The \emph{composite} of ${(Z,R,\phi)\colon \scrG \rightarrow \scrG'}$ with ${(Z',R',\phi')\colon \scrG' \rightarrow \scrG''}$ is given by the submersion ${Z\times_{Y'} Z' \rightarrow Y\times_X Y''}$ and the line bundle $R\otimes R'$.

\begin{definition}
A \emph{connection} on a morphism $F={(Z,R,\phi)\colon\scrG\rightarrow \scrG'}$ of gerbes with connections is a holomorphic connection on $R$ making $\phi$ connection-preserving.
\end{definition}

Any two morphisms $(R,\phi), (\tilde{R},\tilde{\phi})\colon \scrG \rightarrow \scrG'$ differ by a line bundle on the base $X$: the maps $\phi$ and $\tilde{\phi}$ define descent data for the line bundle $R_z\otimes \tilde{R}_z^*$ on $Z$. Hence it is the pullback of a bundle on $X$. Similarly with connections.


\begin{proposition}
Let $F\colon \scrG \to \scrG'$ be a morphism of holomorphic gerbes. Then
\[
	\mathrm{DD}(\scrG)=\mathrm{DD}(\scrG') \in H^3(X;\Z(1)_D).
\]
If the gerbes and the morphism $F$ have connections, their classes in $H^3(X;\Z(2)_D)$ agree.
\end{proposition}

Similarly, for smooth gerbes the topological Dixmier--Douady classes agree \cite[Proposition~3.2]{MR1794295}. As a consequence of the previous proposition we have \cite[Proposition~3.4]{MR1794295}:

\begin{corollary}\label{cor-mur-ste}
Let $\scrG_2=(Y_2,L_2,m)$ be smooth gerbe on $X_2$. Consider a commutative diagram
\[\xymatrix{
Y_1 \ar[r]^\phi\ar[d]_{\pi_1}&Y_2\ar[d]^{\pi_2}\\
X_1\ar[r]^\phi & X_2
}\]
of smooth maps, where $\pi_1$ is a submersion. Then by pullback we obtain a gerbe $\scrG_1=(Y_1,\phi^*L_2,\phi^*m)$ on $X_1$ with topological Dixmier--Douady class $\phi^*\mathrm{DD}^\mathrm{top}(\scrG_2) \in H^3(X_1;\Z)$.
\end{corollary}

\subsection{Transformations}

We now come to the 2-morphisms of our bicategory.

\begin{definition}
Let $F=(Z, R, \phi)$ and $\tilde{F}=(\tilde{Z}, \tilde{R}, \tilde{\phi})$ be morphisms from $\scrG=(Y,L,m)$ to $\scrG'=(Y',L',m')$.
A \emph{transformation} ${\alpha\colon F \Rightarrow \tilde{F}}$ consists of
\begin{enumerate}
\item
A submersion ${(\kappa,\tilde{\kappa})\colon W \rightarrow Z\times_{Y\times Y'} \tilde{Z}}$
\item
A bundle isomorphism $\psi\colon \kappa^*R \rightarrow \tilde{\kappa}^*\tilde{R}$ over $W$.
\end{enumerate}
So $\psi$ induces maps $\psi_w\colon R_{z} \rightarrow \tilde{R}_{\tilde{z}}$, for $(\kappa,\tilde{\kappa})(w)=(z,\tilde{z})$.
For ${(w_1, w_2) \in W^{[2]}}$ let $(\kappa,\tilde{\kappa})(w_k)=(z_k,\tilde{z}_k)$, $(\varsigma,\varsigma')(z_k)=(y_k,y'_k)$. We require the commutativity of
\begin{equation*}
\begin{tikzpicture}[baseline=(current  bounding  box.center)]
\matrix (m) [matrix of math nodes, column sep=3em, row sep=2em]
{
	L_{y_1,y_2}\otimes R_{z_2}	&	R_{z_1}\otimes L'_{y'_1, y'_2}\\
	L_{y_1,y_2}\otimes \tilde{R}_{\tilde{z}_2}	&	\tilde{R}_{\tilde{z}_1}\otimes L'_{y'_1,y'_2}\\
};
\draw[->, font=\scriptsize] (m-1-1) -- node [left] {$\id\otimes\psi$} (m-2-1);
\draw[->, font=\scriptsize] (m-1-1) -- node [above] {$\phi$} (m-1-2);
\draw[->, font=\scriptsize] (m-1-2) -- node [right] {$\psi\otimes \id$} (m-2-2);
\draw[->, font=\scriptsize] (m-2-1) -- node [below] {$\phi$} (m-2-2);
\end{tikzpicture}
\end{equation*}
If the morphisms $F, \tilde{F}$ are equipped with connections, the transformation $\alpha$ is \emph{compatible with the connections} if $\psi$ preserves connections.
\end{definition}

Two transformations are identified if they coincide on a pullback \cite{MR2318389}. This gives a bicategory of gerbes.


\subsection{Further Operations}

Let $\scrG=(Y,L,m)$ be a gerbe on $X$ and let ${f\colon X' \rightarrow X}$ be a holomorphic map. The \emph{pullback gerbe}
$f^*\scrG$ is given by ${Y' = Y\times_X X' \rightarrow X'}$ and the pullback line bundle of $L$.\\

The \emph{tensor product} $\scrG\otimes \scrG'$ of two gerbes $\scrG=(Y,L,m)$, $\scrG'=(Y',L',m')$ on $X$ has the submersion ${Y\times_X Y' \rightarrow X}$ and the exterior product line bundle $L\otimes L'$. Similarly, we have a tensor product of morphisms.
For more details, see \cite{MR2318389}.
The Dixmier--Douady class is compatible with pullback and additive with respect to tensor products. Equipping $f^*\scrG$ and $\scrG\otimes\scrG'$ with the tensor product and pullback connections, this also holds for the $3$-curvature.

\section{The Canonical Gerbe on $\GL(n,\C)$}\label{sec:HolGerbe}

\subsection{Further Notation}\label{prelims-con}

Using the exponential map we transport the order on $[0,2\pi[$ to the unit circle. More generally we shall use the notation $x < y$ for $x,y \in \C^\times$ when $x/|x| < y/|y|$ (only a transitive relation). Hence $x<y$ when the ray through $y$ is obtained from the ray through $x$ by a proper counterclockwise rotation not passing through the positive reals.

Similarly the notation $x\leq y$ includes that $x$ and $y$ may lie on the same ray.

\begin{definition}
Let $x=r_xe^{i\varphi_x}, y = r_y e^{i\varphi_y} \in \C^\times$ be non-zero complex numbers, written in polar coordinates with $0\leq \varphi_y - \varphi_x < 2\pi$. The \emph{sector} of radii $0\leq r < R \leq +\infty$ is
\[
	S^{r,R}(x,y) = \left\{  se^{i\varphi} \mid r < s < R, \varphi_x < \varphi < \varphi_y \right\}.
\]
We write $S(x,y)$ when $r=0, R=+\infty$. The \emph{ray} through $x\in \C^\times$ is the subset $\R_{>0}x$. For $x,y \in \C^-=\C\setminus[0,\infty)$ we define also the \emph{unordered sector}
\[
	S^{r,R}[x,y] = \begin{cases}
		S^{r,R}(x,y)	&	(x\leq y),\\
		S^{r,R}(y,x)	&	(y\leq x).
	\end{cases}
\]
\end{definition}

For a matrix $A\in \GL(n,\C)$ let ${\chi_A= \prod_{\lambda} (X-\lambda)^{n_\lambda}}$ be the characteristic polynomial and write $\operatorname{spec}(A) \subset \C^\times$ for its spectrum. Recall that $\C^n$ may be decomposed as an internal direct sum of the $n_\lambda$-dimensional generalized eigenspaces ${V_\lambda(A) = \ker(A-\lambda\cdot \id)^{n_\lambda}}$.
We say $\lambda \in \C^\times$ is an \emph{eigenray} of $A$ if the ray through $x$ meets $\operatorname{spec}(A)$.

\begin{definition}
For a subset $S\subset \C$ define
\[
	V_S(A) = \bigoplus\{V_\kappa(A) \mid \kappa \in S\}
\]
as the internal direct sum over all subspaces $V_\lambda(A) \subset \C^n$ with $\lambda \in S$. 
\end{definition}

For $S, T, R \subset \C$ with $S\cap T \cap \operatorname{spec}(A)=\emptyset$ and $(S\cup T) \cap \operatorname{spec}(A)=R\cap \operatorname{spec}(A)$
\begin{equation}\label{disj-property}
	V_S(A) \oplus V_T(A)=V_{R}(A).
\end{equation}

The highest exterior power of a finite-dimensional vector space with `$\oplus$' defines a strong monoidal functor $\Lambda^\mathrm{top}$ to one-dimensional vector spaces with `$\otimes$', where $\Lambda^\mathrm{top}(\{0\})=\C$.

\subsection{Construction of the Gerbe}

Let $\C^- = \C\setminus [0,\infty)$ and
\begin{equation}\label{Y}
	Y= \left\{(A,z) \in \GL(n,\C)\times \C^-\;\middle|\;  \text{$z$ not an eigenray of $A$} \right\}.
\end{equation}
Using that eigenvalues are bound by the norm, one shows that the set of all $(A,z) \in \GL(n,\C)\times S^1$ with $\R_{>0}z \cap \operatorname{spec}(A)\neq \emptyset$ is closed. It follows that $Y$ is an open subset, in particular a complex manifold and also the projection
\begin{equation}\label{pi-def}
	\pi\colon Y \rightarrow \GL(n,\C),\quad (A,z)\mapsto A
\end{equation}
is a holomorphic submersion. We define a family of complex vector spaces ${L\rightarrow Y^{[2]}}$ by defining the fiber over $(A,x,y) \in Y^{[2]}=Y\times_{\GL(n,\C)} Y$ as follows:

%
%

\begin{definition}
For $x,y \in \C^\times$ let $\lambda_A(x)= \Lambda^\mathrm{top}(V_{S(1,x)}(A))$ and
\begin{equation}\label{L-def}
	L_{A,x,y} = \lambda_A(x)\otimes \lambda_A(y)^*.
\end{equation}
\end{definition}

The multiplication $m$ of the gerbe is the restriction to $Y^{[3]}$ of the following operation:

\begin{definition}\label{def:multipl}
For $x,y,z \in \C^\times$ we have a bilinear map
\begin{equation}\label{multi-def}
	m\colon L_{A,x,y}\times L_{A,y,z} \to L_{A,x,z},\quad
	m(u\otimes \alpha, v\otimes\beta) = \alpha(v) u\otimes\beta,
\end{equation}
where $u\in\lambda_A(x), \alpha \in \lambda_A(y)^*, v\in \lambda_A(y), \beta\in \lambda_A(z)^*$.
\end{definition}

\begin{lemma}\label{lem:alg-gerbe}
The gerbe multiplication is associative.
\end{lemma}
\begin{proof}
Straight-forward verification:
\begin{align*}
m(m(u\otimes \alpha, v\otimes \beta), w\otimes\gamma)&=m(\alpha(v) u\otimes \beta, w\otimes \gamma) = \alpha(v)\beta(w) u\otimes \gamma,\\
m(u\otimes \alpha,m(v\otimes \beta, w\otimes\gamma)) &=m(u\otimes\alpha, \beta(w) v\otimes\gamma) =
\beta(w)\alpha(v) u\otimes \gamma.
\end{align*}
\end{proof}

\begin{lemma}\label{identific}
We have canonical isomorphisms:
\begin{enumerate}
\item
$L_{A,y,x}\cong L_{A,x,y}^*$.
\item
$L_{A,x,y}\cong \Lambda^\mathrm{top}(V_{S(y,x)})$, provided $x\geq y$ and $y$ is not an eigenray of $A$. The gerbe multiplication corresponds to the external wedge product under this identification.
\item
$L_{A,x,y}\cong \C$, provided $\spec(A)\cap S[x,y]=\emptyset$.
\end{enumerate}
\end{lemma}

\begin{proof}
1. Symmetry isomorphism of `$\otimes$'.
2. Apply \eqref{disj-property} to $\left(S(1,y), S(y,x),S(1,x)\right)$ to get $V_{S(1,y)}\oplus V_{S(y,x)} = V_{S(1,x)}$. The monoidal structure on $\Lambda^\mathrm{top}$ gives $\lambda(y)\otimes \Lambda^\mathrm{top}(V_{S(y,x)}) \cong \lambda(x)$.
3. $V_{S(1,x)}(A)=V_{S(1,y)}(A)$, so $\lambda_A(x) = \lambda_A(y)$.
\end{proof}

We still need to define a complex manifold structure on the total space $L$ for which the bundle becomes locally trivial. For this it suffices to produce holomorphic line bundle structures on the restriction of $L$ over an open cover of $Y^{[2]}$, making sure that on overlaps the corresponding complex structures agree.

\subsection{Holomorphic Structure}

By a \emph{domain} we mean an open bounded subset ${O\subset \C}$ with piecewise smooth boundary. For example, a bounded sector is a domain.

\begin{lemma}\label{tech-lemma123}
Let $A_0 \in \GL(n,\C)$ and let $O$ be a domain. Let $O\subset S$ be a superset satisfying $\spec(A_0)\cap \bar{S} \subset O$. Then there exists a neighborhood $U_0$ of $A_0$ such that:
\begin{enumerate}
\item
The number of eigenvalues $n_A(O)=n_{A_0}(O)$ of $A\in U_0$ in $O$ counted with multiplicity is constant.
\item
$\spec(A)\cap S = \spec(A) \cap O$ \textup($\forall A\in U_0$\textup).
\end{enumerate}
In particular $n_A(S)=n_A(O)=n_{A_0}(O)=n_{A_0}(S)$ and $V_S(A) = V_O(A)$.
\end{lemma}

\begin{proof}
Define $O'$ as the union of all balls $B_r(\lambda) \subset \C\setminus\bar{S}$ around $\lambda \in \spec(A_0) \setminus \bar{S}$. Since $\spec(A_0)$ is finite we may decrease the radii and assume the balls are disjoint and $d(O,O')>0$ and then $O'$ is a domain.
Consider the lower semi-continuous function
\begin{equation}\label{varphidef}
	\varphi\colon \GL(n,\C) \to \R,\quad
	\varphi(A) = \sup_{\lambda \in \partial O\cup \partial O'} |\det(A-\lambda E_n)|.
\end{equation}
Since $\varphi(A_0)>0$ we find a connected open neighborhood $U_0$ with
\begin{equation}\label{choiceU0lemma}
	A_0 \in U_0 \subset \{A\in U \mid \varphi(A) > \varphi(A_0) / 2 > 0\}.
\end{equation}
The zero-counting integral
\begin{equation}\label{zero-count-lemma}
\frac{1}{2\pi i}\oint \frac{\chi'_A(\lambda)}{\chi_A(\lambda)} d\lambda
\end{equation}
is a holomorphic integer-valued function on $U_0$, hence constant. Applied to $\partial O$ we obtain 1.~unless $O=\emptyset$ in which case 1.~is trivial. Applied also to $\partial O'$ we see that for $A\in U_0$
\begin{equation}\label{lem-main-inclusion}
	\spec(A)\subset O\cup O'.
\end{equation}
Then $\spec(A)\cap \bar{S} \subset O$ since $O' \cap \bar{S} = \emptyset$ and with $O \cap \partial S=\emptyset$ we conclude 2. 
\end{proof}

For example, for $\emptyset = O \subset S$ with $\spec(A_0)\cap\bar{S}=\emptyset$ we get near $A_0$ that $\spec(A)\cap S = \emptyset$. More generally, this holds whenever $\spec(A_0)\cap O = \emptyset$.

\begin{lemma}\label{lemmaVO}
Let $O \subset \C$ be a domain. Let $U \subset \GL(n,\C)$ be open with 
\[
	\operatorname{spec}(A) \cap \partial O = \emptyset,\qquad \forall A \in U.
\]
Then $V_O \to U$ is a holomorphic sub vector bundle of the trivial bundle $\C^n$.
\end{lemma}

\begin{proof}
It suffices to show this near each $A_0 \in U$ and to consider the case $O\neq \emptyset$. From the proof of the preceeding lemma for $O=S$ we get $O'$ satisfying \eqref{lem-main-inclusion} for all $A \in U_0 \subset U$. Since ${d(O,O')>0}$ the function $e$ on $O\cup O'$ defined by $e|_O = 1$, $e|_{O'}=0$ is holomorphic. Functional calculus gives us ${e_A \in M_n(\C)}$ depending holomorphically on $A \in U_0$. This is the projection onto the generalized eigenspaces $V_O(A)$ in the support of $e$. Consider
the holomorphic map
\[
	\psi\colon \left( U_0\times \im e_{A_0} \right) \oplus \left( U_0\times \ker e_{A_0} \right) \rightarrow U_0\times \C^n,\quad
	\psi(A,v)=\begin{cases}
	(A,e_A(v)) & v\in \im e_{A_0},\\
	(A,v) & v\in \ker e_{A_0}.
	\end{cases}
\]
Then $\psi|_{A_0}$ is the identity and $\psi$ remains invertible on fibers close to $A_0$. By restricting we get an isomorphism from the trivial bundle $U_0\times \im e_{A_0}$ to $\psi(U_0\times \im e_{A_0})$. By the previous lemma the number of eigenvalues near $A_0$ counted with multiplicity in $O$ is constant.  Since the multiplicity equals the dimension of the corresponding generalized eigenspace, the fiber dimensions in ${\psi(U_0\times \im e_{A_0}) \subset \im(e)}$ coincide, hence ${\psi(U_0\times \im e_{A_0}) = \im(e)}$ and so $\psi|_{U_0\times \im_{A_0}}$ is a trivialization of $\im(e)=V_O$ near $A_0$.
\end{proof}

%

\begin{lemma}
Let $(A_0,x_0,y_0) \in Y^{[2]}$ and let $O \subset \C^-$ be a domain with
\begin{equation}\label{prop-assumption}
\spec(A_0)\cap \bar{O} = \spec(A_0)\cap S[x_0,y_0].
\end{equation}
Then in a neighborhood $W\subset Y^{[2]}$ of $(A_0,x_0,y_0)$ we have
\begin{equation}\label{VOVS}
	V_O(A) = V_{S[x,y]}(A),\qquad \forall (A,x,y)\in W.
\end{equation}
\end{lemma}

\begin{proof}

Since $0\notin \spec(A_0)$ is finite, there are sectors $x_0\in S(x_0^-, x_0^+)$, $y_0 \in S(y_0^-, y_0^+)$ in $\C^-$ containing no eigenvalues of $A_0$. Also $\spec(A_0)$ is contained in an annulus of radii $0<r<R$. Applying Lemma~\ref{tech-lemma123} to $\emptyset \subset S(x_0^-, x_0^+)$ and $\emptyset \subset S(y_0^-, y_0^+)$ we see that in neighborhood of $A_0$ we have
\[
	\spec(A)\cap S(x_0^-, x_0^+)=\emptyset,\quad
	\spec(A)\cap S(y_0^-, y_0^+)=\emptyset.
\]
Apply Lemma~\ref{tech-lemma123} again to $O\subset O\cup S[x_0,y_0]$ to get a possibly smaller neighborhood (note that our assumptions imply $\spec(A_0)\cap \partial O = \emptyset$) in which by Lemma~\ref{tech-lemma123}~2.
\begin{equation}\label{compare-sets}
	\spec(A) \cap S[x_0,y_0] \subset \spec(A)\cap O
\end{equation}
and $n_A(O) = n_{A_0}(O)=n_{A_0}(S[x_0,y_0])$. Passing to a smaller neighborhood, apply Lemma~\ref{tech-lemma123} to $S^{r,R}[x_0,y_0]\subset S[x_0,y_0]$. Then near $A_0$ we have $n_{A_0}(S[x_0,y_0]) = n_A(S[x_0,y_0])$. Hence (the cardinalities in) \eqref{compare-sets} agree. Set $W=U_0 \times S(x_0^-, x_0^+)\times S(y_0^-, y_0^+)$. For $(A,x,y) \in W$ we have $\spec(A)\cap S[x,y]=\spec(A)\cap S[x_0,y_0]=\spec(A)\cap O$ and \eqref{VOVS} follows.
\end{proof}

\begin{definition}\label{def-hol-structure}
Near $(A_0,x_0,y_0)$ with $x_0\geq y_0$ the complex structure is defined as follows. Pick $O\subset \C^-$ satisfying \eqref{prop-assumption}. Then $V_{S[x,y]}(A)=V_O(A)$ for all $(A,x,y)$ in a neighborhood $W$. By Lemma~\textup{\ref{lemmaVO}} the right hand side of the identification of Lemma~\textup{\ref{identific}}
\begin{equation}\label{Lholomorphic}
L|_W\cong \Lambda^\mathrm{top}(V_O)|_W
\end{equation}
is a holomorphic line bundle. Declare \eqref{Lholomorphic} to be a biholomorphism. This gives a well-defined complex structure on $L|_{\{A, x\geq y\}}$, using that \eqref{VOVS} is independent of the choice of $O$. Let $\tau(A,x,y)=(A,y,x)$. Declare also a biholomorphism
\begin{equation}\label{Lholomorphic-dual}
L|_{\{A,x\leq y\}}\cong \tau^*L_{\{A, x\geq y\}}^*.
\end{equation}
On a neighborhood of $(A_0,x_0,y_0)$ with $x_0$ and $y_0$ on the same ray, $L$ is $\Lambda^\mathrm{top}(\{0\})=\C$ or its dual (which are biholomorphic), so \eqref{Lholomorphic}, \eqref{Lholomorphic-dual} define the same complex structure there.
\end{definition}

\begin{proposition}\label{m-holomorphic}
The gerbe multiplication $m$ is holomorphic.
\end{proposition}

\begin{proof}
The problem is local, so fix $(A_0,x_0,y_0,z_0)$. There are six cases to consider. Suppose $x_0\geq y_0\geq z_0$. Choose disjoint domains $O$ and $O'$ whose closures contain precisely those eigenvalues of $A_0$ that belong to $S[x_0,y_0]$ and $S[y_0,z_0]$, respectively. According to Definition~\ref{def-hol-structure}, the canonical identifications of $L$ with $\Lambda^\mathrm{top}(V_O)$ near $(A_0,x_0,y_0)$ and with $\Lambda^\mathrm{top}(V_{O'})$ near $(A_0,y_0,z_0)$ are biholomorphisms. Similarly $O''=O\cup O'$ is a domain with which $L$ is biholomorphic to $\Lambda^\mathrm{top}(V_{O''})$ near $(A_0,x_0,z_0)$. In the commutative diagram
\[\xymatrix{
\pr_{12}^*L \otimes \pr_{23}^*L \ar[d]\ar[r]^-m & \pr_{13}^*L\ar[d]\\
\Lambda^\mathrm{top}(V_O)\otimes \Lambda^\mathrm{top}(V_{O'}) \ar[r]^-{\wedge} & \Lambda^\mathrm{top}(V_{O''})
}\]
the vertical maps are therefore biholomorphisms (all bundles are restricted to a neighborhood of $(A_0,x_0,y_0,z_0)$). Because `$\wedge$' is also holomorphic it follows that $m$ is holomorphic. The remaining five cases are reduced to this one. For example, near $x_0 \geq z_0 \geq y_0$ using the biholomorphism \eqref{Lholomorphic-dual} the multiplication is identified with a fiberwise isomorphism
\begin{equation}\label{mult-case1}
	\pi_{12}^* L \otimes \pi_{32}^* L^* \to \pi_{13}^*L.
\end{equation}
Taking the tensor product of this map with $\id_{\pi_{32}^* L}$ gives the map
\[
	\pi_{12}^* L \otimes \to \pi_{13}^*L\otimes \pi_{32}^* L
\]
which is easily checked to be the inverse of $m$ considered above, which is already known to be holomorphic. It follows that \eqref{mult-case1} is also holomorphic.
\end{proof}

\begin{theorem}\label{main-thm-1}
$\scrG^\can=(Y,L,m)$ defined above in \eqref{pi-def}, \eqref{L-def}, \eqref{multi-def} is a holomorphic gerbe with $\mathrm{DD}(\scrG^\can)$ the canonical generator of $H^3(\GL(n,\C);\Z)$.
\end{theorem}

Corollary~\textup{\ref{cor:DixmierDouadyG}} below will complete the proof of Theorem~\ref{thm:HolomorphicGerbe}.

\begin{proof}
It remains only to compute the Dixmier--Douady class. For this we will compare $\scrG^\can$ with the basic gerbe $\scrG^\mathrm{basic}=(Y_1, E, m_1)$ on $U(n)$ of Murray--Stevenson \cite{MR2463811} using Corollary~\ref{cor-mur-ste}. We first recall their definitions in our notation \cite[p.~7]{MR2463811}:
\[
	Y_1 = \{(A,z) \in U(n)\times S^1 \mid z \notin \spec(A)\cup \{1\}\}
\]
Their line bundle \cite[(3.1)]{MR2463811} is
\[
	E_{(x,y,A)}=  \Lambda^\mathrm{top}\left(V_{S(y,x)}(A)\right)
\]
over the set of $x>y$ with $S(y,x)\cap \spec(A)\neq \emptyset$. For $x,y \in S^1$ with $\spec(A)\cap S[x,y]=\emptyset$ they define $E_{(x,y,A)}=\C$. Finally, for $x<y$ with $S(x,y)\cap \spec(A)\neq \emptyset$ they define
\[
	E_{(x,y,A)}=E_{(y,x,A)}^*.
\]
In \cite[(3.7)]{MR2463811} the gerbe multiplication is the wedge product over $x > y > z$, extended to all of $Y_1^{[3]}$ by dualization. Let $i\colon U(n)\to \GL(n,\C)$ be the inclusion. Our definitions of $L$ and $m$ were designed to avoid cases, but note that by Lemma~\ref{identific}
\[
	i^*L \cong E
\]
and the gerbe multiplication is also given by the wedge product over $x\geq y \geq z$. As discussed in the proof of Proposition~\ref{m-holomorphic}, over other parts of $Y^{[3]}$ our gerbe multiplication is also given by dualization. Hence $\scrG^\mathrm{basic}=\incl^*\scrG^\can$, so from Corollary~\ref{cor-mur-ste} $\mathrm{DD}(\scrG^\mathrm{basic}) = \incl^*\mathrm{DD}(\scrG^\mathrm{can}) $. Since $\incl^*\colon H^3(\GL(n,\C);\Z) \cong H^3(U(n);\Z)$, the claim follows.
\end{proof}

\section{Cohomological Theory on Stein Manifolds}\label{sec:CohomStein}

In this section we collect a number of facts for the Deligne cohomology of Stein manifolds. These will be needed in the proof of Theorem~\ref{theorem:MainMultiplicativity} in the next section.

\subsection{Stein manifolds}


\begin{definition}
A complex Lie group $G$ is a \emph{Stein group} if the underlying manifold is a Stein manifold \textup(see~\textup{\cite[p.~136]{MR2029201}}\textup).
\end{definition}

$\GL(n,\C)$ and any closed complex subgroup of $\GL(n,\C)$ is a Stein group. Any semi-simple connected or simply-connected solvable complex Lie group is Stein.

%
%

\begin{proposition}\label{Prop:isomorphicMor}
Let $X$ be a contractible Stein manifold \textup(for example, a polycylinder\textup) and let $\scrG,\scrG'$ be gerbes on $X$. Then any two morphisms $\scrG\rightarrow \scrG'$ are isomorphic, meaning we find a transformation between them.
\end{proposition}

\begin{proof}
Two stable morphisms differ by a line bundle on $X$, which in our case is a holomorphic line bundle $L$. But $L$ is holomorphically trivial, by the Grauert--Oka Theorem~\cite{MR0098199}, and a trivialization defines a transformation.
\end{proof}

\subsection{Exponential Sequence}
The long exact sequence induced by
\[
	0 \rightarrow \Omega^{ < p}_X [-1] \rightarrow \Z(p)_D \rightarrow \Z \rightarrow 0
\]
with the fact $H^*(\Omega^{<p}) = H^*(X;\C)$ for $*< p-1$ shows that $H^q(\Z(p)_D)$ sits inside a Bockstein sequence. By the Five Lemma we get
\begin{equation}\label{flatCZ}
	H^q(X; \Z(p)_D) \cong H^{q-1}(X;\C/\Z),\qquad 0<q<p-1
\end{equation}
The case $q\geq p$ is more difficult and leads to the \emph{exponential sequence}
\[
	\rightarrow H^{q-1}(X;\Z) \rightarrow H^{q-1}(\Omega^{<p}) \rightarrow H^q(\Z(p)_D) \rightarrow H^q(X;\Z) \rightarrow H^q(\Omega^{<p}) \rightarrow
\]
If $X$ is a Stein manifold and $q>p>0$ then $H^q(\Omega^{<p}) = 0$, by Cartan's~Theorem~B \cite[p.~51]{MR0064154}. Putting this into the exponential sequence gives
\begin{equation}\label{Eqn:DeligneTopological}
	H^q(\Z(p)_D) \cong H^q(X;\Z),\qquad q>p.
\end{equation}

\begin{corollary}\label{cor:DixmierDouadyG}
The class Dixmier--Douady class $\mathrm{DD}(\scrG^\can)$ of the canonical gerbe generates $H^3(\GL(n,\C), \Z(1)_D) \cong H^3(\GL(n,\C), \Z) = \Z$.
\end{corollary}

Since $H^3(\GL(n,\C), \Z(2)_D) \cong H^3(\GL(n,\C), \Z(1)_D)$ there is a unique holomorphic connection on $\scrG^\can$, up to stable isomorphisms with connection. An important point is that this connection may be constructed explicitly by $\C$-linear projection $\C^n \rightarrow V_O(A)$ onto the eigenspace bundles.

\section{Multiplicative Structure}\label{sec:MultiplicativeStructure}

The existence of a multiplicative structure depends only on the stable isomorphism class of the gerbe and is therefore a cohomological problem.

\begin{definition}\label{def:mult-gerbe}
Let $G$ be complex Lie group with product $\mu$.
A \emph{multiplicative holomorphic gerbe} on $G$ consists of a holomorphic gerbe $\scrG$ on $G$, a morphism $M\colon \pr_1^*\scrG \otimes \pr_2^*\scrG \rightarrow \mu^*\scrG$, and a \emph{transformation} $\alpha\colon M\circ(M\otimes\id)\Rightarrow M\circ(\id\otimes M)$. 
%
%
%
The transformation $\alpha$ should fit into the usual coherence pentagon \textup{\cite[p.~47]{MR2610397}}.
A \emph{connection} on a multiplicative gerbe consists of connections on $\scrG$ and $M$ so that $\alpha$ is compatible with the connections.
\end{definition}

In general, there are obstructions to finding a multiplicative structure on a given gerbe. For Stein groups, Theorem~\ref{theorem:MainMultiplicativity}, which we shall prove in this section, asserts that this obstruction for \emph{holomorphic} multiplicative structures reduces to a \emph{topological} problem.

\subsection{Simplicial spaces}

A simplicial space \cite{MR0498552} is a functor
\[
X\colon \Delta^\mathrm{op} \rightarrow \mathbf{Top}.
\]
For example, the constant simplicial space has $X_n=X$ for a fixed space $X$ and all faces and degeneracies are the identity. For a topological group $G$, let $BG_\bullet$ denote the nerve of $G$ with $n$-simplices $BG_n = G^n$. There is a simplicial map $EG_\bullet \rightarrow BG_\bullet$ whose fibers are the constant simplicial spaces $G$, see~\cite{MR0232393}.

\begin{definition}
The \emph{\v{C}ech complex} of a simplicial space $X_\bullet$ with coefficients in a sheaf $\mathcal{F}$ is the total complex $\check{C}^*(X_\bullet;\mathcal{F})$ of the double \v{C}ech complex $C^{pq} = \check{C}^q(X_p, \mathcal{F})$. Given a basepoint $\mathrm{pt}\in X_0$, the \emph{reduced \v{C}ech complex} is $\check{C}^*(X_\bullet; \mathrm{pt};\mathcal{F})=\check{C}^*(X_\bullet;\mathcal{F})/\check{C}^*(\mathrm{pt}_\bullet;\mathcal{F})$.
\end{definition}

On constant simplicial spaces one recovers usual sheaf cohomology. Similarly, for a complex of sheaves $\mathcal{F}^*$ one defines $H(X_\bullet,\mathcal{F}^*)$ using the \v{C}ech hypercomplex~\cite[p.~28]{MR2362847}. From~\cite{MR0232393} we recall that the simplicial cohomology $H(BG_\bullet; \Z)$ computes the cohomology of the classifying space $BG=|BG_\bullet|$ of a Lie group $G$.

\subsection{Multiplicative Extensions}

For a double complex let $F^p$ denote the $p$-th vertical filtration $C^{iq}, i\geq p$. We have a short exact sequence of \v{C}ech cochain complexes
\begin{equation}\label{KES124}
	0\rightarrow F^2 \rightarrow \check{C}^*(BG_\bullet, \mathrm{pt}; \mathcal{F}) \xrightarrow{\tau} \frac{\check{C}^*(BG_\bullet, \mathrm{pt}; \mathcal{F})}{F^2} \rightarrow 0
\end{equation}
Identifying the rightmost term with $\check{C}^{*-1}(G,\mathcal{F})$, the map $\tau$ simply collapses all but the second column of the double complex. In cohomology $\tau$ induces the transgression map, see~\cite[Lemma~2.9]{MR2610397}. Using Lemma~\ref{lem:IteratedCover}, a gerbe with holomorphic connection may be described, up to isomorphism, by a cocycle in $\check{C}^3(G, \Z(2)_D)$. The data for the morphism $M$ and transformation $\alpha$ in Definition~\ref{def:mult-gerbe} corresponds exactly to an extension to a cocycle of the double complex $\check{C}^4(BG_\bullet,\mathrm{pt}; \Z(2)_D)$.

To prove Theorem~\ref{theorem:MainMultiplicativity} it remains therefore only to show:

\begin{lemma}\label{lem:forget-connect}
The map $H^*(BG_\bullet, \Z(2)_D) \rightarrow H^*(BG;\Z)$ is an isomorphism \textup($*>0$\textup).
\end{lemma}

\begin{proof}
By the long exact sequence induced by $0\rightarrow \Omega^{<2}[-1] \rightarrow \Z(2)_D \rightarrow \Z \rightarrow 0$ it suffices to show $H^*(BG_\bullet, \Omega^{<2}) = 0$. To do this, we consider the spectral sequence for simplicial spaces $E_1^{pq} = H^q(BG_p, \Omega^{<2}) \Rightarrow H^{p+q}(BG_\bullet; \Omega^{<2})$.

From Cartan's Theorem B {\cite[p.~51]{MR0064154}} and the sheaf hypercohomology spectral sequence we deduce for any Stein manifold $X$ that the space $H^*(X, \Omega^{<2})$ is $H^0(X;\C)$ for $*=0$, $H^1(X;\C)$ for $*=1$, and zero else.

It follows that $E_1^{pq}=0$ unless $q=0,1$. The first row $E_1^{p0}=H^0(G^p;\C)$ is given by $\C\xrightarrow{0} \C \xrightarrow{1} \C \rightarrow \cdots$, so exact apart from the first term. The second row $E_1^{p1}$ reads
\begin{equation}\label{differentials-exact}
0\rightarrow H^1(G;\C) \xrightarrow{\delta} H^1(G^2;\C) \xrightarrow{\delta} H^1(G^3;\C) \xrightarrow{\delta} \cdots
\end{equation}
where the maps are $\delta=\sum_i (-1)^i d_i^*$ induced by the face maps $d_i$ of the nerve. Identify $H^1(G^n)=H^1(G)^{\oplus n}$. Recall that for any topological group $(\pi_1(G,1),+)$ is abelian and
that $\pi_1(\mu)$ is the sum. It follows that $\mu^* \colon H^1(G) \to H^1(G)^{\oplus 2}$ is $\mu^*(x)=(x,x)$. From this one computes that the differentials $\delta\colon H^1(G)^{\oplus n} \to H^1(G)^{\oplus(n+1)}$ in \eqref{differentials-exact} are
\[
	\delta(x_1,\ldots,x_n) = \begin{cases}
	n\text{ even}:&	(-x_1,0,x_2-x_3,0,\ldots,x_{n-2}-x_{n-1},0,x_n),\\
	n\text{ odd}:&	(0,x_2,x_2,x_4,x_4,\ldots, x_{n-1},x_{n-1},0).
	\end{cases}
\]
Hence $E_1^{p1}$ is exact. We conclude $E_2^{pq}=0$ for $p+q>0$, whence the result.
\end{proof}

\begin{remark}
Lemma~\ref{lem:forget-connect} holds also for semisimple complex Lie groups \cite[Theorem~5.11]{MR1291698}. This relies on facts for the Hodge filtration established by Deligne \cite{MR0498552}.
\end{remark}

For $\mathcal{F}=\Z(2)_D$ the sequence \eqref{KES124} induces in cohomology the long exact sequence
\[
	\cdots \to H^4(BG;\Z(2)_D) \xrightarrow{\tau} H^3(G;\Z(2)_D) \to H^3(F^2) \to \cdots 
\]

\begin{definition}
A choice of preimage of $\mathrm{DD}(\scrG)\in H^3(G;\Z(2)_D)$ under the transgression map $\tau$ is the \emph{multiplicative class} $\lambda(\scrG) \in H^4(BG_\bullet;\Z(2)_D)$ of the multiplicative gerbe with holomorphic connection \textup(see~\textup{\cite{MR2610397}} in the smooth category\textup).
\end{definition}

From the proof above, it is clear that the multiplicative class determines the multiplicative structure, up to (multiplicative) isomorphism, see~\cite{MR2610397}. From Theorem~\ref{theorem:MainMultiplicativity} we now deduce:

\begin{corollary}\label{GLNcorollary}
For $G=\GL(n,\C)$ every holomorphic gerbe with connection admits a multiplicative structure.\end{corollary}

\begin{proof}
From Lemma~\ref{lem:forget-connect} we deduce
\[
H^4(B\GL(n,\C)_\bullet, \Z(2)_D) \cong H^4(B\GL(n,\C);\Z)=\Z c_1^2 \oplus \Z c_2.
\]
The preimages $c_1^B \in H^2(BG_\bullet;\Z(1)_D), c_2^B \in H^4(BG_\bullet;\Z(2)_D)$ in Deligne cohomology of these classes are the \emph{Be{\u\i}linson--Chern classes}. On the other hand, by \eqref{Eqn:DeligneTopological}
\[
H^3(G;\Z(2)_D) \cong H^3(G;\Z) = \Z.
\]
The topological transgression map takes $c_2$ to the generator of $\Z$ (see~\cite{MR0051508}). It follows that every class $[\scrG] \in H^3(G;\Z(2)_D)$ is in the image of $\tau$. 
\end{proof}

\section{Holomorphic 2-Gerbes}\label{sec:HigherGerbes}

\subsection{2-Gerbes}\label{ssec:2gerbe}
Our definition of $2$-gerbe is weaker than that in~\cite{MR2032513}, where it is demanded that the multiplication functor $M$ be strictly associative. We only require associativity up to coherence transformations.

\begin{definition}
A \emph{holomorphic $2$-gerbe} $\mathfrak{G}=(\rho, \scrG, M, \alpha)$ on $X$ consists of
\begin{enumerate}
\item
A holomorphic submersion $\rho\colon V\rightarrow X$.
\item
A holomorphic gerbe $\scrG$ on $V^{[2]}$.
\item
A morphism of gerbes $M\colon \pr_{12}^*\scrG \otimes \pr_{23}^*\scrG \rightarrow \pr_{13}^*\scrG$ over $V^{[3]}$.
\item
An associativity transformation $\alpha\colon M\circ(M\otimes\id)\Rightarrow M\circ(\id\otimes M)$ between the two composite morphisms
\[
\pr_{12}^*\scrG\otimes\pr_{23}^*\scrG \otimes \pr_{34}^*\scrG \rightarrow \pr_{14}^*\scrG
\]
\end{enumerate}
The transformations $\alpha$ are required to fit into the usual commutative pentagon \textup{\cite[p.~47]{MR2610397}}. In detail, this means the commutativity of diagram~\eqref{pentagon} below.
\end{definition}

\begin{remark}
Let us unwind this definition and see that is entails an $X$-indexed family of $2$-categories.
The elements $v \in \rho^{-1}(x)$ are the \emph{objects} at $x\in X$. The gerbe $\scrG$ contains a submersion $\pi=(\pi_1,\pi_2)\colon Y \rightarrow V^{[2]}$ whose fibers $y\in \pi^{-1}(v_1,v_2)$ are the \emph{1-arrows} from $v_1$ to $v_2$. Given two $1$-arrows $y_1, y_2$ we have a complex line $L_{y_1,y_2}$ of {$2$-arrows}. The multiplication in the gerbe $\scrG$ gives a strictly associative \emph{vertical composition} of $2$-arrows.
Horizontal composition is encoded in the morphism $M$. It includes a submersion $Z \rightarrow {(Y{}_{\pi_2}\!\times_{\pi_1} Y)\times_{V^{[2]}} Y}$ onto the set of composable triangles of $1$-arrows.
We regard $z\in Z$ mapping to $(y_{\alpha\beta}, y_{\beta\gamma}, y_{\alpha\gamma})$ as a \emph{filler} of this triangle and then think of $y_{\alpha\gamma}$ as a choice of horizontal composition of $y_{\alpha\beta}$ and $y_{\beta\gamma}$. In particular, it is not unique. Given two filled triangles $z\mapsto(y_{\alpha\beta}, y_{\beta\gamma}, y_{\alpha\gamma})$ and $\tilde{z}\mapsto (\tilde{y}_{\alpha\beta},\tilde{y}_{\beta\gamma}, \tilde{y}_{\alpha\gamma})$ on the same vertices,  $M$ provides us with isomorphisms
\[
	L_{y_{\alpha\beta},\tilde{y}_{\alpha\beta}}\otimes L_{y_{\beta\gamma},\tilde{y}_{\beta\gamma}} \otimes R_{\tilde{z}} \rightarrow
	R_{z} \otimes L_{y_{\alpha\gamma},\tilde{y}_{\alpha\gamma}}.
\]
This is the horizontal composition of $2$-arrows.
The morphism $\alpha$ includes a submersion $W\rightarrow{(Z\times_Y Z) \times_{Y^{[4]}} (Z\times_Y Z)}$ to fillers $(z_{\beta\gamma\delta},z_{\alpha\gamma\delta},z_{\alpha\beta\delta},z_{\alpha\beta\gamma})$ of 
diagrams:
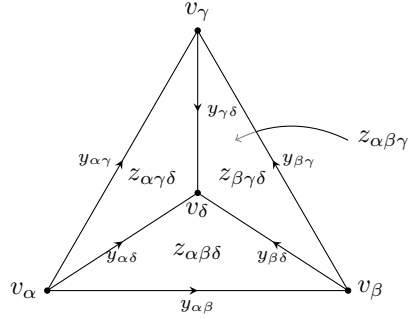
\begin{equation*}
\begin{tikzpicture}[decoration={
    markings,
    mark=at position 0.5 with {\arrow{stealth}}
    }]
\draw (-1,0);    
\draw[->] (4,2) node[right] {$z_{\alpha\beta\gamma}$} .. controls (3.5,2.25) and (3,2.25) .. (2.5,2);
\filldraw[white,opacity=0.5] (0,0) -- (4,0) -- (2,3.46) -- (0,0);
\draw[font=\scriptsize,postaction={decorate}] (0,0) -- node[below]{$y_{\alpha\beta}$} (4,0);
\draw[font=\scriptsize,postaction={decorate}] (4,0) -- node[right]{$y_{\beta\gamma}$} (2,3.46);
\draw[font=\scriptsize,postaction={decorate}] (0,0) -- node[left]{$y_{\alpha\gamma}$} (2,3.46);
\draw[font=\scriptsize,postaction={decorate}] (0,0) -- node[below]{$y_{\alpha\delta}$} (2,1.3);
\draw[font=\scriptsize,postaction={decorate}] (2,3.46) -- node[right]{$y_{\gamma\delta}$} (2,1.3);
\draw[font=\scriptsize,postaction={decorate}] (4,0) -- node[below]{$y_{\beta\delta}$} (2,1.3);
\path (2,0.5) node{$z_{\alpha\beta\delta}$};
\path (1.4,1.5) node{$z_{\alpha\gamma\delta}$};
\path (2.6,1.5) node{$z_{\beta\gamma\delta}$};

\filldraw (0,0) circle (1pt) node[left] {$v_\alpha$};
\filldraw (4,0) circle (1pt) node[right] {$v_\beta$};
\filldraw (2,3.46) circle (1pt) node[above] {$v_\gamma$};
\filldraw (2,1.3) circle (1pt) node[below] {$v_\delta$};

\end{tikzpicture}
\end{equation*}
%
%
%
%
%
%
%
We regard a preimage $w_{\alpha\beta\gamma\delta}$ as a filler of this tetrahedron. $\alpha$ gives an isomorphism
\begin{equation}\label{assoc-iso}
	\psi_{w_{\alpha\beta\gamma\delta}}\colon	R_{z_{\alpha\beta\delta}}\otimes R_{z_{\beta\gamma\delta}} \rightarrow R_{z_{\alpha\gamma\delta}} \otimes R_{z_{\alpha\beta\gamma}}.
\end{equation}
These mediate between the two ways to horizontally compose three $2$-arrows. 
The commutative pentagon mentioned above is then expressed as follows. Let $\alpha,\beta,\gamma,\delta,\varepsilon$ be objects with arrows $y_{\alpha\beta},\ldots, y_{\delta\varepsilon}$ between them, let $z_{\alpha\beta\gamma},\ldots, z_{\gamma\delta\varepsilon}$ be fillers of all triangles, and let $w_{\beta\gamma\delta\varepsilon}, w_{\alpha\gamma\delta\varepsilon}, w_{\alpha\beta\delta\varepsilon}, w_{\alpha\beta\gamma\varepsilon}, w_{\alpha\beta\gamma\delta}$ be  fillers of the resulting tetrahedra. For every such data, we have a commutative diagram:
\begin{equation}\label{pentagon}
\begin{tikzpicture}[baseline=(current  bounding  box.center)]
\matrix (m) [matrix of math nodes, column sep=-1em, row sep=2em]
{
	& R_{\alpha\beta\gamma}\otimes R_{\alpha\gamma\delta}\otimes R_{\alpha\delta\varepsilon} &\\
	R_{\beta\gamma\delta}\otimes R_{\alpha\beta\delta} \otimes R_{\alpha\delta\varepsilon}&&R_{\alpha\beta\gamma}\otimes R_{\gamma\delta\varepsilon}\otimes R_{\alpha\gamma\varepsilon}\\
	R_{\beta\gamma\delta}\otimes R_{\beta\delta\varepsilon}\otimes R_{\alpha\beta\varepsilon}&& R_{\beta\gamma\varepsilon}\otimes R_{\gamma\delta\varepsilon}\otimes R_{\alpha\beta\varepsilon}\\
};
\draw[->, font=\scriptsize] (m-1-2) -- node [right] {$\;\;\id\otimes\psi_{\alpha\gamma\delta\varepsilon}$} (m-2-3);
\draw[->, font=\scriptsize] (m-2-3) -- node [right] {$\id\otimes\psi_{\alpha\beta\gamma\varepsilon}$} (m-3-3);
\draw[->, font=\scriptsize] (m-3-1) -- node [above] {$\psi_{\beta\gamma\delta\varepsilon}\otimes\id$} (m-3-3);
\draw[->, font=\scriptsize] (m-1-2) -- node [left] {$\psi_{\alpha\beta\gamma\delta}\otimes\id\;\;$} (m-2-1);
\draw[->, font=\scriptsize] (m-2-1) -- node [left] {$\id\otimes\psi_{\alpha\beta\delta\varepsilon}$} (m-3-1);
\end{tikzpicture}
\end{equation}
\end{remark}

\subsection{The Dixmier--Douady Class}\label{ssec:DD2gerbe}

Let $\mathfrak{G}=(\rho,\scrG,M,\alpha)$ be a $2$-gerbe on $X$. By Lemma~\ref{lem:IteratedCover} we find  open covers $\{U_\alpha\}$ of $X$ with holomorphic sections as follows: 
\begin{enumerate}
\item
$v_\alpha\colon U_\alpha \rightarrow V$ of $\rho$.
\item
$y_{\alpha\beta} \colon U_{\alpha\beta} \rightarrow Y$ of the pullback of $\pi$ along $(v_\alpha,v_\beta)\colon U_{\alpha\beta}\rightarrow V^{[2]}$.
\item
$z_{\alpha\beta\gamma}\colon U_{\alpha\beta\gamma}\rightarrow Z$ of the pullback of $Z\rightarrow (Y\times_V Y)\times_{V^{[2]}} Y$ along $(y_{\alpha\beta},y_{\beta\gamma},y_{\alpha\gamma})$.
\item
$w_{\alpha\beta\gamma\delta}\colon U_{\alpha\beta\gamma\delta} \rightarrow W$ of the pullback of ${W\rightarrow (Z\times_Y Z) \times_{Y^{[4]}} (Z\times_Y Z)}$ along\\$(z_{\beta\gamma\delta},z_{\alpha\beta\delta},z_{\alpha\beta\gamma},z_{\alpha\gamma\delta})\colon U_{\alpha\beta\gamma\delta} \rightarrow Z^4$.
\item
Trivializations of $z_{\alpha\beta\gamma}^*R$
\end{enumerate}

Since both sides are trivialized, \eqref{assoc-iso} is just a holomorphic function $g_{\alpha\beta\gamma\delta} \in \O^*(U_{\alpha\beta\gamma\delta})$.

\begin{definition}
The \emph{Dixmier--Douady class} of $\mathfrak{G}$ is
$[(g_{\alpha\beta\gamma\delta})]=\mathrm{DD}(\mathfrak{G}) \in \check{H}^4(X,\Z(1)_D)$.
\end{definition}

\begin{definition}
Let $\mathfrak{G}=(\rho, \scrG, M, \alpha)$ be a holomorphic $2$-gerbe. A \emph{connection} on $\mathfrak{G}$ consists of a connection $\nabla^L$ on the gerbe $\scrG$ and a connection $\nabla^R$ on the morphism $M$. The transformation $\alpha$ is required to be compatible with the connections.
\end{definition}

Write the connection on $z_{\alpha\beta\gamma}^*R$ as $d+A_{\alpha\beta\gamma}$. Since \eqref{assoc-iso} preserves connections,
$(g_{\alpha\beta\gamma\delta},A_{\alpha\beta\gamma})$ refines the Dixmier-Douady class to $H^4(X, \Z(2)_D)$.


\subsection{The Canonical $2$-Gerbe}\label{ssec:canonical2gerbe}

For a holomorphic vector bundle $E$ over an algebraic manifold $X$ we have the Be{\u\i}linson--Chern classes (see~\cite{MR760999,MR1667678,MR944991})
\[
	c_p^B(E) \in H^{2p}(X, \Z(p)_D).
\]
By \cite[Proposition~8.2]{MR944991} these are characterized by functoriality and by the requirement that they map to the  Chern classes via $H^{2p}(X,\Z(p)_D) \rightarrow H^{2p}(X;\Z)$. The construction of such classes is based on Deligne's theory of mixed Hodge structures \cite{MR0498552}. Our Lemma~\ref{lem:forget-connect} above shows that for $p\leq 2$ one may also deduce the existence of these classes from the theory of Stein spaces.

Recall from Theorem~\ref{thm:HolomorphicGerbe} the canonical gerbe $\scrG^\can$ on $\GL(n,\C)$. Using Corollary~\ref{GLNcorollary} we equip this gerbe with the multiplicative structure with multiplicative class $c_2$.

\begin{definition}
Let $E\rightarrow X$ be a holomorphic vector bundle. Its \emph{associated $2$-gerbe} $\mathfrak{G}(E)$ is defined by the following data:
\begin{enumerate}
\item
As submersion $\rho$ we take the principal bundle of frames ${P_\GL(E) \rightarrow X}$.
\item
The gerbe $\mathcal{G}(E) = \delta^*\mathcal{G}^\can$ with ${\delta\colon P_\GL^{[2]}\rightarrow \GL(n,\C)}$ given by ${p \delta(p,q)=q}$.
\item
$M$ and $\alpha$ are pulled back from the multiplicative structure on $\scrG^\can$.
\end{enumerate}
For $E=T^*X$ we call $\mathfrak{G}(T^*X)$ the \emph{canonical $2$-gerbe} of the complex manifold.
\end{definition}


\begin{proof}[Proof of Theorem~\textup{\ref{thm:Canonical2Gerbe}}]
Functoriality is obvious. By the definition of the topological Dixmier--Douady class of a $2$-gerbe one sees easily that it is the pullback of the topological multiplicative class $\lambda(\scrG^\can) = c_2 \in H^4(B\GL(n,\C);\Z)$ under the classifying map $X\rightarrow B\GL(n,\C)$. Alternatively, one may appeal to the smooth case \cite{MR2610397}. This proves (2). Now (3) follows from (1), (2) and \cite[Proposition~8.2]{MR944991}.
\end{proof}

\section{An Application and Further Outlook}\label{sec:appl}

The following statement makes essential use of the holomorphic structure on $\scrG^\can$ and the integrability of the complex structure on $X$.

\begin{theorem}\label{appl1}
Let $X$ be a closed complex $6$-manifold with ${b_3=0}$, ${c_2(X)=0}$ \textup(e.g.~when $H^4(X;\Z)=0$\textup). Then the canonical $2$-gerbe ${\mathfrak{G}(T^*X)}$ of $X$ is trivial.
\end{theorem}

\begin{proof}
We consider the long exact sequence in sheaf cohomology
\[
	\cdots \rightarrow H^3(X;\Z) \rightarrow H^3(X,\O) \rightarrow H^3(X, \O^*) \rightarrow H^4(X;\Z) \rightarrow \cdots
\]
By definition, $H^3(X,\O^*)=H^4(X,\Z(1)_D)$. Moreover, $H^3(X,\O)=H^{0,3}(X)$ which by Serre duality may be identified with $H^{0,3}(X)=H^0(X, \Lambda^3(X))$, the holomorphic sections of the canonical bundle of $X$. Observe that the linear map
\[
	H^{3}(X, \Lambda^3(X)) \rightarrow H^3_{dR}(X)
\]
is injective (it is well-defined since $\partial \omega = 0$ by reasons of degree and $\bar\partial \omega=0$ since $\omega$ is holomorphic). Indeed, suppose $\omega = d \eta$ is a holomorphic $3$-form that bounds. In a chart neighborhood $(z^1,z^2,z^3)$ we may write $\omega = f dz^1dz^2dz^3$. Since $M$ is closed,
\[
	\int_M |f|^2 d\mathrm{vol}=\int_M \omega\wedge \bar\omega = \int_M d(\eta \wedge \bar\omega) = \int_{\partial M} \eta \wedge \bar\omega = 0.
\]
This implies $f=0$ on an arbitrary chart neighborhood, so $\omega=0$.
Now our assumptions imply $H^3(X,\O)=0$, so that by the exact sequence $H^3(X,\O^*)$ is mapped injectively into $H^4(X;\Z)$. Since the canonical $2$-gerbe is mapped to $c_2(X)$, the claim follows.
\end{proof}

The triviality of the canonical line bundle has the topological consequence that the middle Betti number is non-zero. We ask whether there is a similar topological obstructions against the triviality of the canonical $2$-gerbe.


\bibliographystyle{plain}

\end{document}